\documentclass{amsart}

\usepackage[utf8]{inputenc} % Required for inputting international characters
\usepackage[T1]{fontenc} % Output font encoding for international characters
\usepackage[backend=bibtex,style=numeric,natbib=true,firstinits=true,doi=false,isbn=false,url=false,eprint=false,maxbibnames=99]{biblatex} % Use the bibtex backend with the authoryear citation style (which resembles APA)
\usepackage[autostyle=true]{csquotes} % Required to generate language-dependent quotes in the bibliography
\usepackage[shortlabels]{enumitem}
\usepackage{amsmath}
\usepackage{amsfonts}
\usepackage{amsthm}
\usepackage{tikz-cd}
\usepackage{tikz}
\usepackage{geometry}
\usetikzlibrary{matrix,arrows}

\theoremstyle{plain}
\newtheorem{Th}{Theorem}[section]
\newtheorem{Lemma}[Th]{Lemma}

\newtheorem{Prop}[Th]{Proposition}
\newtheorem{PropDef}[Th]{Proposition/Definition}

\theoremstyle{definition}
\newtheorem{Def}[Th]{Definition}
\newtheorem{Conj}[Th]{Conjecture}
\newtheorem{Rem}[Th]{Remark}
\newtheorem{?}[Th]{Problem}
\newtheorem{Ex}[Th]{Example}

\newcommand{\im}{\operatorname{im}}

\numberwithin{equation}{section}
\newcommand{\inv}{^{-1}}
\newcommand{\PP}{\mathbb{P}}

\newcommand{\ZZ}{{\mathbb Z}}
\newcommand{\NN}{{\mathbb N}}
\newcommand{\QQ}{{\mathbb Q}}

\newcommand{\CC}{{\mathbb C}}

\newcommand{\Hrm}{\mathrm{H}}

\DeclareMathOperator{\Proj}{Proj}

\DeclareMathOperator{\Aut}{Aut}
\DeclareMathOperator{\Spec}{Spec}

\newcommand{\pr}{^\prime}

\newcommand{\Hom}{\mathrm{Hom}}

\newcommand{\Pic}{\textrm{Pic}}

\newcommand{\Oc}{\mathcal{O}}

\newcommand{\Ac}{\mathcal{A}}

\newcommand{\Fc}{\mathcal{F}}

\newcommand{\Hc}{\mathcal{H}}

\newcommand{\Xc}{\mathcal{X}}

\newcommand{\Xm}{\mathcal{X}}

\newcommand{\OKs}{\mathit{O}_{K,S}}

\newcommand{\oQ}{\overline{\mathbb{Q}}}

\newcommand{\ok}{\overline{k}}

\newcommand{\sbe}{\subseteq}
\newcommand{\spe}{\supseteq}

\newcommand{\OX}{{\mathcal{O}_{X}}}

\newcommand{\OY}{{\mathcal{O}_{Y}}}

\newcommand{\kb}{{k\left(b\right)}}

\DeclareMathOperator{\rank}{rank}

\DeclareMathOperator{\Homology}{H}

\author{Philipp Licht}
\address{Philipp Licht \\
	Institut f\"{u}r Mathematik\\
	Johannes Gutenberg-Universit\"{a}t Mainz\\
	Staudingerweg 9, 55099 Mainz\\
	Germany.}
\email{plicht05@uni-mainz.de}

 %hyperbolicity


\begin{document}

\title[Hyperbolicity of the moduli of Fano threefolds of Picard rank 1, index 1 and degree 4]{Hyperbolicity of the moduli of certain Fano threefolds}
\maketitle

\begin{abstract}
	We prove the Shafarevich conjecture for Fano threefolds of Picard rank 1, index 1 and degree 4. %We show that the moduli of these Fano threefolds is stratified via the loci of hyperelliptic Fano threefolds and smooth quartics. 
	
\end{abstract}

%----------------------------------------------------------------------------------------
% Introduction
%----------------------------------------------------------------------------------------

\section{Introduction}

In \cite{Faltings1983} Faltings proved that, given a number field $K$, 
a finite set of places $S$ on $K$ and a positive integer $g$, there are only finitely many isomorphism classes  of abelian varieties of dimension $g$ over $K$ with good reduction outside $S$.  Faltings's finiteness result illustrates a more general phenomenon, commonly referred to as the Shafarevich conjecture, that the set of objects of fixed type over a number field $K$ with good reduction outside $S$ should be finite. 
This far-reaching conjecture has been verified for    K3 surfaces (and hyperk\"ahler varieties) 
\cite{Yves1996, Lie2022, She2017, Teppei2020}, 
cyclic covers \cite{JavanpeykarLoughranMathur},
ample hypersurfaces in abelian varieties \cite{LawrenceSawin} (building on \cite{LawrenceVenkatesh}),
polycurves \cite{IppeiNagamachi2019, Javanpeykar2015PolarizedSurfaces},
flag varieties \cite{Javanpeykar2015Flag}, 
complete intersections of Hodge level at most one \cite{JaLoCI2017},
certain types of Fano threefolds \cite{JLo18}, 
del Pezzo surfaces \cite{Scholl1985}, 
and 
Enriques surfaces \cite{Takamatsu2020}.

In this paper, we focus on the Shafarevich conjecture for Fano threefolds. Interestingly, the Shafarevich conjecture for Fano threefolds can fail; see \cite[Theorem~1.4]{JLo18} for a precise statement. Our main result however verifies  the Shafarevich conjecture for (smooth) Fano threefolds  with Picard rank 1, index 1 and degree 4 (see Definition \ref{d:Fano variety} and Definition \ref{d:type}). Note that this is one of the cases for which the Shafarevich conjecture  was not handled by Loughran and Javanpeykar \cite{JLo18}.

\begin{Th}[Shafarevich conjecture for Fano threefolds of type (1,1,4)]\label{t:Shaferevic for Fanos}
	Let $K$ be a number field and let $S$ be a finite set of places on $K$. Then the set of $K$-isomorphism classes of Fano threefolds of Picard rank 1, index 1 and degree 4 over $K$ with good reduction outside $S$ is finite. 
\end{Th}

The Shafarevich conjecture is a statement about the finiteness of integral points on certain moduli stacks. The rough idea is that sets of integral points on moduli stacks should be finite by Lang's conjecture \cite{JBook, Lang1986} if the stack is hyperbolic (in a suitable sense).  

In fact, to prove Theorem \ref{t:Shaferevic for Fanos}, we translate the claimed finiteness statement into a property of the moduli stack of such Fano threefolds. 
 It is therefore crucial to understand the structure of the moduli of such Fano threefolds.  
After Iskovskikh's classification \cite[table~3.5]{Isk80},   the Fano threefolds of Picard rank $1$, index $1$ and degree $4$  come in two types:
\begin{enumerate}[a)]
	\item smooth quartics, and
	\item double covers of smooth quadrics in $\mathbb{P}^4$.
\end{enumerate}
The latter type is called hyperelliptic. We show the folklore fact that in families of such Fano varieties, the locus of hyperelliptic ones is closed.

\begin{Th}[Closedness of the hyperelliptic locus]
	\label{t:Main result ii}
	Let $B$ be a scheme over $ \ZZ[1/2] $. Let $X \to B$ be a Fano threefold of type $(1,1,4)$ over $B$. Then the set of $b\in B$  such that $ X_b $ is hyperelliptic is closed in $B$. 
\end{Th}

The proof of Theorem \ref{t:Main result ii}  uses  the fact that Fano threefolds of type $(1,1,4)$ are   weighted complete intersections, locally for the Zariski topology. This is well-known for such Fano threefolds over algebraically closed fields, and our contribution is that this also persists over more general base schemes.

\begin{Th}[Zariski local presentation]\label{t:local_presentation_for_Fanos}
	Let $B$ be a scheme over $ \ZZ[1/2] $ and let $ f \colon X \to B $ be a Fano threefold of type (1,1,4).
	Let $b\in B$ be a point. There is an affine open neighbourhood $ U = \Spec R \sbe B $ of $b$ and polynomials $q_2,q_4 \in R[x_0,\dots,x_4,z]$ of degree $\deg(q_2) = 2$ and $ \deg(q_4) = 4$  with respect to the grading  given by $ \deg(x_i) = 1 $ and $ \deg(z) = 2$ such that the   Fano threefold $ f\vert_U\colon f\inv(U) \to U $ is isomorphic to $ V_+(q_2,q_4) \sbe \PP_R(1,1,1,1,1,2) $, and $\omega_{f\inv(U)/U}\inv$ is identified with $ \Oc(1) $. 
\end{Th}

Let $ \mathcal{F}$ be the moduli stack of Fano threefolds of type $(1,1,4)$, let $ \mathcal{Q} \sbe \mathcal{F}$ be the locus of smooth quartics and let $ \mathcal{H} \sbe \mathcal{F} $ be the locus of hyperelliptic Fano threefolds; see Section \ref{s:stack of Fano threefolds} and Section \ref{s:stacks of hyperelliptics and quartics} for definitions.  Theorem \ref{t:Main result ii} implies that $ \Fc $ is stratified via $\Hc $ and $ \mathcal{Q}$.

\begin{Th}[Stratification of the moduli of Fano threefolds of type (1,1,4)] \label{t:stratification of Fano stack} 
	The inclusion $ \mathcal{H} \to \mathcal{F} $ is a   closed immersion, the inclusion $ \mathcal{Q} \to \mathcal{F}$ is an open immersion, and  for any field $k$ with $2\in k^\times$, we have 
	$$
	\mathcal{F}(k) = \mathcal{Q}(k) \sqcup \mathcal{H}(k).
	$$
\end{Th}

Theorem \ref{t:Shaferevic for Fanos} is a consequence of the fact that $\mathcal{F}$ is an "arithmetically hyperbolic" stack.  Here we follow \cite{JLChevalleyWeil2020} and say that a finitely presented algebraic stack $X$ over an algebraically closed field  $k$ of characteristic $0$ is   \emph{arithmetically hyperbolic over $k$} 
if there is a $\ZZ$-finitely generated subring $A \sbe k$ and a (finitely presented) model $\mathcal{X}$ for $X$ over $A$ such that, for any $\ZZ$-finitely generated subring $A'\sbe k$ containing $A$, the set 
$$ \im \left(\pi_0(\Xc(A')) \to \pi_0(\Xc(k))\right)$$ 
is finite.

We will show that the stack $\mathcal{F}$ is  arithmetically hyperbolic over any algebraically closed field of characteristic zero.  We follow the terminology of \cite{Javanpeykar2022} and say that a finitely presented algebraic stack $X$ over an algebraically closed field  $k$ of characteristic $0$ is   \emph{absolutely arithmetically hyperbolic} if $ X_L$ is arithmetically hyperbolic over $L$ for every algebraically closed field extension $ L \spe k $.

 For an arbitrary stack, 
there should be no difference between being arithmetically hyperbolic over $\overline{\mathbb{Q}}$ and being absolutely arithmetically hyperbolic, i.e.,   arithmetic hyperbolicity  should persist over all field extensions. This is formalized by the following conjecture for stacks alluded to in \cite[Remark~4.13]{JLChevalleyWeil2020}.

\begin{Conj}[Persistence Conjecture]
	Let $k$ be an algebraically closed field of characteristic zero, and let $X$ be a finitely presented algebraic stack over $k$.
	If $X$ is arithmetically hyperbolic over $k$, then $X$ is absolutely arithmetically hyperbolic. 
\end{Conj}

This is a stacky version of \cite[Conjecture 1.1]{Javanpeykar2021} (see also \cite[Conjecture~17.5]{JBook}). The conjecture  says, in particular, that the finiteness of rational points over number fields on a projective variety over $\mathbb{Q}$ should imply the finiteness over all finitely generated fields; this was formulated as a precise question by Lang in \cite[p.~202]{Lang1986}.  
The main result of this paper is that   $\Fc$ is absolutely arithmetically hyperbolic.

\begin{Th}[Main result]\label{t:Fano stack is arithmetically hyperbolic}
	The stack $ \mathcal{F} $ of Fano threefolds of type $(1,1,4)$ is absolutely arithmetically hyperbolic.
\end{Th}

The proof of this result uses the fact that the intermediate Jacobian defines a morphism of stacks
$$
	p\colon  (\mathcal{F}_\CC)^{an} \to (\Ac_{30,\CC})^{an}, 
$$
where $\Ac_{30}$ is the stack of 30-dimensional principally polarized abelian schemes.
We show that this period map is quasi-finite  by using  the stratification of $\mathcal{F}$ (Theorem \ref{t:stratification of Fano stack}) and invoking our results on the infinitesimal Torelli property for Fano threefolds of type (1,1,4) \cite[Theorem~1.3]{LichtTorelli2022}. 
%The stratification allows us to verify the quasi-finiteness of the period map restricted to each stratum separately. 
%Then the quasi-finiteness is a consequence of the injectivity of differential of the period map; see \
The absolute arithmetic hyperbolicity of $ \mathcal{F} $ will then follow from the absolute arithmetic hyperbolicity of $ \Ac_{30} $.

\subsection*{Acknowledgements}

I would like to thank Ariyan Javanpeykar. He introduced me to Shafarevich conjecture and the moduli of Fano-threefolds. I am very grateful for many inspiring discussion and his help in writing this article. I gratefully acknowledge the support of
SFB/Transregio 45.
%----------------------------------------------------------------------------------------
%	The stack of Fanothreefolds
%----------------------------------------------------------------------------------------

\section{The stack of Fano threefolds}\label{s:stack of Fano threefolds}

Following \cite[Section~2-3]{JLo18}, we introduce the stack of Fano threefolds. 
\begin{Def}\label{d:Fano variety}
	Let $k$ be a field. A \emph{Fano variety over $k$} is a smooth proper geometrically integral variety $S$ over $k$ such that its anticanonical bundle is ample. 
	
	Let $B$ be a scheme. A \emph{Fano scheme over $B$} (or \emph{family of Fano varieties over $B$}) is a smooth proper morphism $X \to B $ of schemes whose fibres are Fano varieties. A \emph{Fano $n$-fold over B}  is a Fano scheme of relative dimension $n$. 
\end{Def}

\begin{Def}\label{d:good reduction} Let $B$ be a Dedekind scheme with function field $K$.  A Fano variety  $X$ over $ K $ has \emph{good reduction over $B$} if there is a Fano scheme $ \Xm \to B$ and an isomorphism $ \Xm_{K} \cong X $.   
\end{Def}

\begin{Rem}
	Let $f\colon X \to B$ be a smooth proper morphism with geometrically connected fibres of relative dimension $n$. Then $ \omega_{X/B} = \wedge^n \Omega_{X/B} $ is a line bundle. Since $f$ is proper, the line bundle $ \omega_{X/B}^{-1} $ is relatively ample if and only if it is fibre-wise ample; see \cite[Theorem~4.7.1]{Gro61}. Hence $ X \to B$ is a Fano  scheme if and only if $ \omega_{X/B}\inv $ is relatively ample.
\end{Rem}

\begin{Def}\label{d:type}
	Let $k$ be an algebraically closed field and let $ X $ be a Fano threefold over $k$. Then we define: \begin{enumerate}
		\item The \emph{Picard rank of $X$} is $ \rho(X) = \mathrm{rank}_\ZZ \,\Pic\, X$.
		\item The \emph{index of $X$} is $r(X) = \max\{ m\in \NN \, \vert \, \omega_X^{-1}/m \in \Pic(X) \}. $
		\item The \emph{degree of $X$} is the triple intersection number $ d(X) = (\omega_X^{-1}/r(X))^3$.
		\item The \emph{type of $X$} is the triple $ (\rho(X),r(X),d(X))$.
	\end{enumerate}
\end{Def}

\begin{Def}
	We define a fibred category $ p\colon  \mathrm{Fano} \to \mathrm{Sch} $, where for a scheme $B$, the objects of $ \mathrm{Fano}(B)$ are Fano threefolds over $B$. 
	A morphism 
	$
	( f\colon X \to B) \to (f\pr\colon X\pr \to B\pr)
	$   
	of two Fano three  in $\mathrm{Fano}$ is given by a pair $ (g,h)$, where $g\colon B \to B\pr $ and $ h\colon X \to X' $ are morphisms of schemes such that the square 
	$$
	\begin{tikzcd}
		X \ar[r,"h"] \ar[d,"f"] & X' \ar[d,"f'"] \\
		B \ar[r,"g"] & B' 
	\end{tikzcd}
	$$
	is cartesian. The functor $p$ is the forgetful functor, that remembers only the base scheme.
	
	Given a triple of positive integers $(\rho,r,d) \in \NN^3$, we define $ \mathrm{Fano}_{\rho,r,d} \to \mathrm{Sch} $ to be the full fibred subcategory of Fano threefolds $f\colon X \to B$ such that all geometric fibres of $f$ are Fano varieties of Picard rank $\rho$, index $r$ and degree $d$. 
	%Likewise $ \mathrm{Fano}_{\rho,*,*} \to \mathrm{Sch} $ is the full fibered subcategroy of Fano threefolds of Piccard number $\rho$. 
	We define $\mathcal{F} $ to be the fibred category  $\left(\mathrm{Fano}_{1,1,4}\right)_{\ZZ[1/2]} $.
\end{Def}

\begin{Prop}\quad \label{P:facts about the stack of Fanos}
	\begin{enumerate}
		\item The fibred category $\mathcal{F}$ is a finite type algebraic stack with an affine diagonal  over $ \ZZ[1/2] $.
		\item The stack $ \mathcal{F}_\QQ $ is smooth over $\QQ$.
		\item There is an $N\in \NN $ such that $ \mathcal{F}_{\ZZ[1/N]} $ is separated over $ \ZZ[1/N] $.
		\item There is an $N\in \NN $ such that for all schemes $ B $ over $ \ZZ[1/N] $ and $ X,Y \in \mathcal{F}(B) $, the morphism $\mathrm{Isom}_B(X,Y) \to B$ is finite.
		\item The stack $ \mathcal{F}_{\QQ} $ is a Deligne-Mumford stack.
	\end{enumerate}
	
\end{Prop}

\begin{proof}
	Note that (1), (2), (3) and (4) follows from  \cite[Lemma~3.5, 3.6, 3.7, 3.8]{JLo18}.  Finally, (5) follows from the fact that a finite type separated algebraic stack with an affine diagonal over $\mathbb{Q}$ is Deligne-Mumford.
\end{proof}

%----------------------------------------------------------------------------------------
%	Fanothreefolds over fields
%----------------------------------------------------------------------------------------

\section{Fano threefolds of Picard number 1, index 1, degree 4 over fields}

\label{s:fano_threefolds_over_fields}
Fano threefolds over an algebraically closed field have been classified by Iskovskikh \cite[Table~3.5]{Isk80} in characteristic $0$. His result was later generalized by Shepherd-Barron \cite{ShB97} to positive characteristic in the case where the Picard rank is $1$. In this section, we will generalize the characterization for Fano threefolds of type $(1,1,4)$ to the case of non algebraically closed fields of characteristic not equal to $2$. In Section \ref{s:families of Fanos}, we will then generalize it to families of Fano threefolds.

\begin{Th}[Iskovskikh-Shepherd-Barron {\cite[Propositions~4.1+4.3]{ShB97}}] 
	Let $k$ be an algebraically closed field. Let $ X$ be a Fano threefold of type $(1,1,4)$ over $k$. 
	\begin{enumerate}[a)]
		\item Then $X$ is a smooth quartic in $ \PP^4_k$, or
		\item $X$ is a double cover of a smooth quadric in $\PP^4_k$ ramified along a smooth surface of degree $8$.
	\end{enumerate}
\end{Th}

Let $k$ be a field with $2\in k^\times$. We consider the weighted projective space
$$
\PP_k(1,1,1,1,1,2) = \Proj k[x_0,\dots,x_4,z]
$$
with weights $ \deg(x_i) = 1 $ for $ i \in \{0,\dots,4\}$ and $ \deg(z) = 2$; see \cite{Dol82} for an introduction on weighted projective varieties. 
Varieties of both types a) and b) represent special cases of smooth weighted complete intersections of degree $(2,4)$.
Considering type a), let $$ X = V_+(q_4) \sbe \PP_k^4  = \Proj k[x_0,\dots,x_4] $$ be a smooth quartic. Then $X$ is isomorphic to the complete intersection 
$$
X' = V_+(z,q_4) \sbe \PP_k(1,1,1,1,1,2). 
$$
Considering type b), let $ X$ be a double cover of a smooth quadric $ V_+(q_2) \sbe \PP^4$ which is ramified along a smooth surface $V_+(q_2,q_4)$ of degree $8$. Then $X $ is isomorphic to 
$$
X' = V_+(q_2,q_4-z^2) \sbe \PP_k(1,1,1,1,1,2).
$$ 
The double cover map is given via projection onto the first $5$ homogeneous coordinates.
Note in both cases $ X' $ does not contain the point $Q=(0:0:0:0:0:1)$. In fact, the following more general statement holds.

\begin{Lemma} \label{l:smooth subvarieties don't contain Q}
	If $ Y \sbe  \PP_k(1,1,1,1,1,2) $ be a positive-dimensional smooth complete intersection, then $Y$ does not contain $Q$. 
\end{Lemma}

\begin{proof}
	The point $Q$ is contained in the open affine neighbourhood 
	$$ U = D_+(z) = \Spec k \left[ \frac{x_i x_j}{z} \middle\vert 0 \leq i \leq j \leq 4 \right] \sbe \PP_k(1,1,1,1,1,2). $$
	The generators $ w_{i,j} = \frac{x_i x_j}{z} $ satisfy relations $$ q_{\alpha,\sigma} = w_{\alpha_1,\alpha_2}w_{\alpha_3,\alpha_4} -w_{ \alpha_{\sigma 1},\alpha_{\sigma 2}}w_{\alpha_{\sigma 3},\alpha_{\sigma 4}},$$
	where $ \alpha \in  \{ 0,\dots, 4\}^4 $ with 
	$\alpha_1 \leq \alpha_2$ and $\alpha_3 \leq \alpha_4$ and $ \sigma $ is a permutation of $ \{1,2,3,4 \}$ 
	with $  \alpha_{\sigma 1} \leq \alpha_{\sigma 2}$ 
	and $\alpha_{\sigma 3} \leq \alpha_{\sigma 4}$. 
	We see $ U \cong \Spec k[w_{i,j}]/(q_{\alpha,\sigma})$. The point  $Q$ corresponds to the point with coordinates $ w_{i,j} = 0$. Note that $ \frac{\partial q_{\alpha,\sigma}}{\partial w_{i,j}}(Q) = 0$ for all $i,j,\alpha,\sigma$.
	Therefore it is not possible for $ U \cap Y $  to satisfy the Jacobi criterion if $Y$ does contain $Q$.
\end{proof}

In the following, we will find an explicit description for the anti-canonical bundle of the weighted complete intersection $X'$.
By {\cite[Theorem~3.3.4]{Dol82}},
there is an isomorphism 
\begin{equation}\label{e:canonical bundle}
	\omega_{X'/k} \cong \Oc_{X'}(-1). 
\end{equation}
Note that for a closed subvariety $ Y \sbe \PP_k(W_1,\dots,W_r) $ with weighted coordinate ring $A$, the sheaf $ \Oc_Y(m) $ is the graded $ \OY$-module associated to the degree shifted graded module $ A(m) $.
In general, it is not true that $\OY(m)$ is a line bundle or that $ \Oc_Y(m) \otimes \Oc_Y(l) \cong \OY(m+l) $; see \cite[Section~1.5]{Dol82}.
For example, if $ P = \PP_k(1,1,1,1,1,2) $ as above, then $ \Oc_P(1) \otimes \Oc_P(1) \not \cong \Oc_P(2) $. 
This is because in any neighbourhood of $Q = (0:0:0:0:0:1)$, the section $z$ of $\Oc_P(2)$
is not a product of sections of $\Oc_P(1) $. However, this is the only problematic point. 

\begin{Lemma}\label{l:O(l) x O(m) = O(l+m)}
	Let $k$ be a field, let $ X \sbe \PP_k(1,1,1,1,1,2) $  be a closed subvariety that does not contain the point $Q = (0:0:0:0:0:1)$. Then for any $ m,l\in \ZZ $, the sheaf $ \OX(m) $ is a line bundle and the multiplication map induces an isomorphism  $$ \OX(m) \otimes \OX(l) \cong \OX(m+l).$$
\end{Lemma}

\begin{proof}
	We have $ X\sbe \PP_k(1,1,1,1,1,2) \setminus \{Q\} = \bigcup D_+(x_i)$, where  
	$$ D_+(x_i) = \Spec k \left[ \frac{x_0}{x_i}, \dots, \frac{x_4}{x_i}, \frac{z}{x_i^2} \right].$$
	Let $ U_i = X \cap D_+(x_i) $. The assertion follows, since the multiplication map 
	$$ 
	\OX\vert_{U_i} \xrightarrow{\cdot x_i^m} \OX(m)\vert_{U_i} 
	$$
	is an isomorphism.  
\end{proof}

By Lemma \ref{l:smooth subvarieties don't contain Q} and Lemma \ref{l:O(l) x O(m) = O(l+m)}, the Isomorphism (\ref{e:canonical bundle}) induces an isomorphism
\begin{equation}\label{e:anti-canonical bundle}
	\omega_{X'/k}^{-i} \cong \Oc_{X'}(i)
\end{equation}
for all $i \in \ZZ$.
The proof of the following proposition will utilize the fact that our Fano threefolds come with a canonical embedding into weighted projective space associated to the anticanonical bundle.

\begin{Prop}\label{p:classifying_Fanos_over_nonclosed_field}
	Let $k$ be a field with $2\in k^\times$, let $X$ be a Fano threefold of type $(1,1,4)$ over $k$.
	Then the function  $\phi(i) = \dim_k \Homology^0(X,\omega_{X/k}^{-i} )  $ satisfies $ \phi(1) = 5, \phi(2)=15, \phi(3) = 35$ and $ \phi(4) = 69$.  
	Let $ \xi_0,\dots,\xi_4 $ be a basis for $\Homology^0(X,\omega_{X/k}^{-1} )$  . Then the following statements hold.
	\begin{enumerate}
		\item If $X_{\bar{k}} $ is of type a), then the variety $X$ is a smooth quartic in $ \PP^4_k$. The monomials of degree $4$, $ \xi_0^4,\xi_0^3\xi_1,\dots,\xi_4^4 $ satisfy a relation $ q_4 $. The map $ \xi_i \mapsto x_i$
		induces an isomorphism of graded $k$-algebras $$ \bigoplus_{i\geq 0} \Homology^0(X, \omega_{X/k}^{-i} ) \cong k[x_0,\dots,x_4]/(q_4) $$
		and an isomorphism
		$
		X \cong V_+(q_4) \sbe \PP_k^4
		$
		of varieties. 
		\item If $X_{\bar{k}} $ is of type b), then the variety $X$ is a double cover of a smooth quadric in $\PP^4_k$ ramified along a smooth surface of degree $8$. The monomials of degree $2$, $ \xi_0^2,\xi_0\xi_1,\dots,\xi_4^2 $ satisfy a relation $ q_2 $ and span a 14-dimensional subspace of $ \Homology^0(X, \omega_{X/k}^{-2} ) $. Let $ \zeta \in \Homology^0(X, \omega_{X/k}^{-2} ) $ be an element completing those monomials to a generating set. 
		There are polynomials $ q_2 \in k[x_0,\dots, x_4]_2$ and $ q_4 \in k[x_0,\dots, x_4,z]_4 $  
		such that the map $ \xi_i \mapsto x_i, \zeta \mapsto z $ induces an isomorphism if graded $k$-algebras 
		$$ \bigoplus_{i\geq 0} \Homology^0(X, \omega_{X/k}^{-i} ) \cong k[x_0,\dots,x_4,z]/(q_2,q_4-z^2) $$
		and an isomorphism
		$
		X \cong V_+(q_2,q_4-z^2) \sbe \PP_k(1,1,1,1,1,2)
		$
		of varieties. 
	\end{enumerate}
\end{Prop}

\begin{proof} We will prove (2). The proof of (1) is similar.
	%Note that we already proved the case where $k$ is algebraically closed in the discussion above.
	Let $X$  be a Fano threefold over $k$ of type (1,1,4) such that there is an isomorphism $$ X_{\ok} \cong V_+(q_2',q_4'-z^2) \sbe \PP_{\ok}(1,1,1,1,1,2), $$
	where $q_j' \in k[x_0,\dots,x_4]_j$. Isomorphism (\ref{e:anti-canonical bundle}) induces an isomorphism $$ \bigoplus_{i \geq 0} \Homology^0\left(X_{\ok}, \omega_{X_{\ok}/\ok}^{-i} \right) \cong k[x_0,\dots,x_4,z]/(q_2',q_4'-z^2) =: R $$
	of graded $k$-algebras. Furthermore for each $ i \geq 0$, there is an isomorphism 
	$$
	\Homology^0(X, \omega_{X/k}^{-i} ) \otimes  \ok \cong \Homology^0\left(X_{\ok}, \omega_{X_{\ok}/\ok}^{-i} \right). 
	$$
	Hence, the value $ \phi(i) $ can be determined by counting the elements of a $\ok$-vector space basis for the degree $i$ part $ R_i $; see \cite{LichtPdh2022} for the computations.
	The basis $ \xi_0,\dots,\xi_4 $ of $ \Homology^0(X, \omega_{X/k}^{-1} ) $ is also a basis for $ \Homology^0(X_{\bar{k}}, \omega_{X_{\bar{k}}/\bar{k}}^{-1} ) $. The image of the multiplication pairing $ R_1 \otimes R_1 \to R_2 $ is generated by the 15 monomials $ x_0^2, x_0x_1, \dots, x_4^2 $. These monomials satisfy the relation $q_2'$. Therefore, the image has dimension 14. Hence, also the monomials $ \xi_0^2,\xi_0\xi_1,\dots,\xi_4^2 $ generate a 14-dimensional subspace of $ \Homology^0(X, \omega_{X/k}^{-2} ) $, and therefore satisfy a relation. 
	
	Let $ q_2 \in k[x_0,\dots,x_4]_2 \setminus \{0\} $ such that $q_2(\xi_0, \dots, \xi_4) = 0$, 
	and let $\zeta \in \Homology^0(X, \omega_{X/k}^{-2} ) $ be an element completing those monomials to a generating set.
	The vector space $(k[x_0,\dots,x_4,z]/(q_2))_4 $ has dimension 70. Hence the map $$ (k[x_0,\dots,x_4,z]/(q_2))_4  \to \Homology^0(X, \omega_{X/k}^{-4} ) $$
	given by $ x_i \mapsto \xi_i, z \mapsto \zeta $ has a non-zero kernel.
	Let $ q''_4 \neq 0 $ be an element of the kernel.
	Then the maps
	$$
	\bigoplus_{i\geq 0} \Homology^0(X, \omega_{X/k}^{-i} ) \to k[x_0,\dots,x_4,z]/(q_2,q''_4)
	$$ 
	and 
	$$
	X \to V_+(q_2,q''_4) \sbe \PP_k(1,1,1,1,1,2)
	$$
	are isomorphisms after the flat base change to the algebraic closure. Therefore, they were isomorphisms, to begin with. By Lemma \ref{l:smooth subvarieties don't contain Q}, the variety $V_+(q_2,q''_4)$ does not contain $Q$. Hence, $q_4''$ has a $ z^2$-term and after a suitable change of coordinates, we get $q_4'' = q_4 - z^2$.
\end{proof}

%----------------------------------------------------------------------------------------
%	Families Fanothreefolds 
%----------------------------------------------------------------------------------------

\section{Families of Fano threefolds of Picard number 1, index 1, degree 4}
\label{s:families of Fanos}

As seen in Section \ref{s:fano_threefolds_over_fields}, over fields of characteristic unequal to $2$, there are two mutually exclusive types of Fano threefolds of Picard rank 1, index 1 and degree 4, namely smooth quartics and hyperelliptic ones. In this section, we will study how these types behave in families of Fano threefolds over schemes $B$ over $ \ZZ[1/2]$. The main tool in this study will be Theorem \ref{t:local_presentation_for_Fanos}, which says that families of such Fano threefolds Zariski locally are weighted complete intersections. 

\subsection*{Local presentation as weighted complete intersection}

We need the following technical lemma, which is a combination of Proposition \ref{p:classifying_Fanos_over_nonclosed_field} with a "cohomology and basechange" argument.

\begin{Lemma}\label{Lemma:fano_anticanonical_bundle_is_free} Let $B$ be a scheme over $ \ZZ[1/2] $ and let $ f\colon X \to B $ be a Fano threefold of type (1,1,4) and let $l\in \ZZ_{>0}$. Then the following statements hold. 
	\begin{enumerate}
		\item The formation of $ f_* \omega_{X/B}^{-l} $ commutes with arbitrary base change.
		\item The sheaf $ f_* \omega_{X/B}^{-l} $ is locally free.
		\item Let $ r(l) = \rank f_* \omega_{X/B}^{-l} $. We have $r(1) = 5$, $r(2)=5$, $r(3)= 35$ and $r(4)= 69$.
		\item Assume additionally that $B$ is affine.  Let $b\in B$ be a point. Then the natural map
		$$
		\Homology^0(X,\omega_{X/B}^{-l}) \otimes \kb \to \Homology^0\left(X_{b},\omega^{-l}_{X_{b}/ k(b)}\right)\colon \quad \xi \otimes \lambda \mapsto \lambda \cdot \xi\vert_{b}
		$$
		is an isomorphism.
	\end{enumerate}
\end{Lemma}

\begin{proof}
	The morphism $ f $ is proper and smooth of relative dimension $3$. 
	Hence $f$ is locally of finite presentation and $ \omega_{X/B}^{-l}$ is a sheaf of finite presentation on $X$ and flat over $B$.  
	By Proposition \ref{p:classifying_Fanos_over_nonclosed_field}, for any point $b\in B $, the fibre $ X_{b} $ is a weighted complete intersection of degree $(2,4)$ in $ \PP_{k(b)}(1,1,1,1,1,2) $,
	and $\omega_{X/B}^{-l}\vert_{X_b}$ is isomorphic to $ \Oc_{X_{b}}(l) $; see Isomorphism (\ref{e:anti-canonical bundle}).
	Hence  $h^1\left(X_{b},\omega_{X_{b}/k(b)}^{-l} \right) = 0 $ and the function $$b \to h^0\left(X_{b},\omega_{X_{b}/k(b)}^{-l}\right)$$ is constant; see Proposition \ref{p:classifying_Fanos_over_nonclosed_field} and \cite[Lemma~6.4]{LichtTorelli2022}.
	By applying \cite[Lemma~1.1.5]{Ben12}, we get (1) and (2). 
	We get (4) from (1) by considering the base change along $ \Spec k(b) \to B $. Assertion (3) is a consequence of (4) and  Proposition \ref{p:classifying_Fanos_over_nonclosed_field}.
\end{proof}

\begin{proof}[Proof of Theorem \ref{t:local_presentation_for_Fanos}]
	We will successively shrink $B$ and thereby also $X$ to find a suitable $U$. By Lemma \ref{Lemma:fano_anticanonical_bundle_is_free}.(2), we can shrink $B$ such that $ B = \Spec R $ is affine and $ f_* \omega_{X/B}^{-l} $ is free for $l\in \{1,2,3,4\}$. We choose a basis $\xi_0, \dots, \xi_4 $ for the free $R$-module $ \Hrm^0(X,\omega_{X/B}^{-1}) = \Hrm^0(B,f_* \omega_{X/B}^{-1})  $.
	Then the restrictions $\xi_0\vert_b, \dots, \xi_4\vert_{b} $ form a basis for $\Homology^0\left(X_{b},\omega^{-1}_{X_{b}}\right) $ by Lemma \ref{Lemma:fano_anticanonical_bundle_is_free}.(4). By Proposition \ref{p:classifying_Fanos_over_nonclosed_field}, we know that the $15$ elements 
	$
	\xi_0\vert_{b}^2, \xi_0\xi_1\vert_{b}, \dots, \xi_4\vert_{b}^2
	$
	generate a subspace of dimension at least $14$ in the $15$-dimensional vector space $\Homology^0\left(X_{b},\omega^{-2}_{X_{b}/ \kb}\right) $.
	We choose $ \tilde{z} \in \Homology^0\left(X_{b},\omega^{-2}_{X_{b}/ k(b)}\right)\setminus\{0\} $ such that $\xi_0\vert_{b}^2, \xi_0\xi_1\vert_{b}, \dots, \xi_4\vert_{b}^2, \tilde{z} $
	is a set of generators for $ \Homology^0\left(X_{b},\omega^{-2}_{X_{b}/ k(b)}\right) $. After shrinking $B$, we may find a $ \zeta \in \Hrm^0(X,\omega_{X/B}^{-2}) $ such that 
	$ \zeta\vert_{b} = \tilde{z} $. Shrinking $B$ again, we may assume that 
	$
	\xi_0^2, \xi_0\xi_1, \dots, \xi_4^2, \zeta 
	$
	generates $\Hrm^0(X,\omega_{X/B}^{-2})$. Hence these elements define a surjection 
	$ \phi\colon R^{16}  \to \Hrm^0(X,\omega_{X/B}^{-2}) $.
	Let $K = \ker(\phi) $.
	As $ \Hrm^0(X,\omega_{X/B}^{-2}) $ is a free $R$-module of rank $15$, we get an exact sequence 
	$$
	0 \to K \otimes k(b) \to k(b)^{16} \to \Homology^0\left(X_{b},\omega^{-2}_{X_{b}/ k(b)}\right) \to 0,
	$$   
	where $ K \otimes k(b) $ is a $1$-dimensional $ k(b)$ vector space. After shrinking $B$, we can find $ \lambda \in K $ sucht that $ \lambda\vert_{b} $ generates $ K \otimes k(b)$.
	In particular $ \lambda\vert_{b} \neq 0 $. We shrink $B$ such that $ \lambda\vert_{p} \neq 0 $ for all $p\in B$.
	Then $ \lambda\vert_{p} $ generates $K \otimes k(p)$ for all $p \in P$. The element $\lambda $ corresponds to a polynomial $ q_2 \in R[x_0,\dots,x_4,z]_2 $ such that $ q_2(\xi_0,\dots,\xi_4,\zeta) = 0$.
	
	By Proposition \ref{p:classifying_Fanos_over_nonclosed_field}, the
	restrictions
	$$
	\xi_0^4\vert_p, \xi_0^3\vert_p\xi_1\vert_p, \dots, \xi_4^4\vert_p, \xi_i\vert_p \xi_j\vert_p \zeta\vert_p, \zeta\vert_{p}^2 
	$$
	generate $\Homology^0\left(X_{p},\omega^{-4}_{X_{p}/ k(p)}\right)$. The map 
	$$
	\left( \frac{ k(p)[x_0,\dots,x_4,z]}{q_2\vert_p} \right)_4
	\to \Homology^0\left(X_{p},\omega^{-4}_{X_{p}/ k(p)}\right)
	$$
	that maps $ x_i\mapsto \xi_i\vert_p$ and $z\mapsto \zeta\vert_p $ has a $1$ dimensional kernel. As above, after shrinking $B$, we find a global generator $\nu$ for this kernel. The section $\nu$ corresponds to a polynomial
	$ q_4 \in R[x_0,\dots,x_4,z]_4 $ such that $ q_4(\xi_0,\dots,\xi_4,\zeta) = 0$.
	The map $ x_i\mapsto \xi_i$, $z\mapsto \zeta$ induces a morphism 
	$$
	X \to V_+(q_2,q_4) \sbe \PP_R(1,1,1,1,1,2).
	$$
	Again by Proposition \ref{p:classifying_Fanos_over_nonclosed_field}, this is an isomorphism fibre-wise. Hence it is an isomorphism.
\end{proof}

\subsection*{The hyperelliptic and smooth quartic locus}

\begin{Def}
	Let $B$ be a scheme over $ \ZZ[1/2] $. A Fano scheme $f\colon X \to B$ of type (1,1,4) is called \emph{hyperelliptic} if there is an open affine cover $ B = \bigcup_{i\in I} U_i$, $U_i = \Spec R_i $ such that the polarized scheme $$ (f\vert_{f\inv(U_i)}\colon f\inv(U_i) \to U_i, \omega_{f\inv(U_i)/U_i}\inv) $$ is isomorphic to a complete intersection $ (V_+(q_2,q_4),\Oc(1)) $ in $ \PP_{R_i}(1,1,1,1,1,2)$ with \\ $ q_2 \in R_i[x_0,\dots,x_4]$.
	
\end{Def}

\begin{Rem} Theorem \ref{t:local_presentation_for_Fanos} guarantees that there is an open affine cover on which a Fano scheme is a weighted complete intersection. The additional condition imposed by being hyperelliptic is that  $ q_2 $ does not depend on $z$. Note that every geometric fibre of a hyperelliptic Fano threefold is hyperelliptic in the sense of Iskovskik's classification. More generally, the following statement holds.
	
	Let $B$ be a scheme over $ \ZZ[1/2] $ and let $f\colon X \to B$ be a hyperelliptic Fano threefold of type (1,1,4).
	Then for every field $k$ and every point $ b \in B(k)$,
	the fibre $ X_b $ is a double cover of a smooth quadric in $\PP^4_k$ ramified along a smooth surface of degree $8$. 
	Indeed, 
	since $f\colon X \to B$ is hyperelliptic, each fibre is a weighted complete intersection $ V_+(q_2,q_4') \sbe \PP_{k}(1,1,1,1,1,2)  $ with $q_2$ not depending on $z$. By Lemma \ref{l:smooth subvarieties don't contain Q}, the polynomial $ q_4' $ has a $z^2$-term and we may assume that $ q_4' = q_4 - z^2$ for some $q_4 \in k[x_0,\dots,x_4] $ after a suitable change of coordinates. The variety $ V_+(q_2,q_4-z^2)$ is a double cover of $V_+(q_2) \sbe \PP^4_k $ ramified along $ V_+(q_2,q_4)$.
\end{Rem}

The condition on $q_2$ in the definition of hyperelliptic does not depend on the choice of the open covering. 

\begin{Lemma}\label{L:definition_does_not_depend_on_cover}
	Let $B$ be a scheme over $\ZZ[1/2] $. Let $f\colon X \to B$ be a hyperelliptic Fano scheme of type (1,1,4). Let $U = \Spec R \subset B $ be an open affine subscheme such that there is an isomorphism of polarized schemes 
	$$ (f\vert_{U}\colon f\inv(U), \omega_{f\inv(U_i)/U}) \cong (V_+(q_2,q_4),\Oc(1))$$ over $U$. 
	Write $ q_2 = q_2' + c z $ with $ q_2' \in R[x_0,\dots,x_4]$. Then we have $c=0$.
\end{Lemma}

\begin{proof}
	Note that $c=0$ if and only if there is an open affine covering $ U = \bigcup V_j $ such that $ c\vert_{V_j} = 0$. 
	As $f\colon X \to B$ is hyperelliptic, after shrinking $U$ if necessary, we may assume that  $  f\inv(U) $ has a presentation  
	$ V_+(q_2'',q_4'') \sbe \PP_R(1,1,1,1,1,2) $ 
	with $ q_j'' \in R[x_0,\dots,x_4]_j$. Thus, there is an isomorphism over $U$ of polarized schemes
	$$
	(V_+(q_2'+c,q_4),\Oc(1)) \cong  (V_+(q_2'',q_4''),\Oc(1)). 
	$$
	The isomorphism of polarized schemes induces an isomorphism 
	$$
	\alpha\colon R[x_0,\dots,x_4,z]/(q_2'+cz,q_4) \to R[x_0,\dots,x_4,z]/(q_2'',q_4'')
	$$
	of graded $R$-algebras.
	Since $z$ is of degree 2, the image $\alpha(z) $  has to be of degree $2$. We write $ \alpha(z) = a z + p $ with $ a \in R$ and  $p \in  R[x_0,\dots,x_4]_2 $. As the degree 2 relation $q_2''$ does not depend on $z$, it follows from the subjectivity of $\alpha$ that $a$ is invertible. Since $ a\cdot c = 0$, it follows that $c= 0$, as required.	
\end{proof}

\begin{Def}\label{d:smooth quartic}
	Let $B$ be a scheme over $ \ZZ[1/2] $. A Fano scheme $f\colon X \to B$  of type (1,1,4) is a \emph{smooth quartic} if every geometric fibre is a smooth quartic. 
\end{Def}

One might ask why we did not define hyperelliptic Fano threefolds via a fibre-wise criterion as well. The problem with such a definition is that over non-reduced bases, it would allow for "hyperelliptic" Fano threefolds that are not double covers, as the following example shows.

\begin{Ex}
	Let $k$ be a field with $2\in k^\times $. Consider the family $$ X = V_+(x_0^2+\dots + x_4^2+\epsilon z,x_0^4+\dots+x_4^4-z^2) \sbe \PP_{k[\epsilon]}(1,1,1,1,1,2) $$ over $ k[\epsilon] := k[t]/(t^2)$.
	This family has just one geometric fibre, which is hyperelliptic. However, it is neither hyperelliptic nor a smooth quartic. 
\end{Ex}

At the same time, this is an example of a first-order deformation of a hyperelliptic Fano threefold that is not hyperelliptic. This is in accordance with the fact that the locus of hyperelliptic Fanos is a closed substack of $ \Fc $. 

\begin{PropDef}[hyperelliptic locus]\label{p:hyperelliptic_locus}
	Let $B$ be a scheme over $ \ZZ[1/2] $ and let $ f\colon X \to B $ be a Fano threefold of type (1,1,4). Then there is a closed subscheme $ c\colon B_{hyp} \to  B $ such that $ X \times_B B_{hyp} \to B_{hyp}$ is hyperelliptic with the following universal property.
	If $ g\colon T \to B $ is a morphism of schemes such that $ X \times_B T$ is hyperelliptic, then there is a unique morphism $ g\pr\colon T \to B_{hyp} $ such that $ g = c\circ g\pr$. 
	Moreover, $B_{hyp}$ is the set of points $ b\in B$ such that the geometric fibre $ X_{\bar{b}} $ over $k(\bar{b})$ is hyperelliptic.  We call $B_{hyp}$ the \emph{hyperelliptic locus of $f\colon X \to B$}.
\end{PropDef}

\begin{proof}
	The universal property shows that if it exists, the hyperelliptic locus is unique up to a unique isomorphism.
	Therefore, the assertion is local on $B$, because the universality of the hyperelliptic locus allows us to construct the locus locally and glue it together afterwards. 
	By Theorem \ref{t:local_presentation_for_Fanos}, we may therefore assume that $ B = \Spec R $ is affine and $ X = V_+(q_2,q_4) \sbe \PP_R(1,1,1,1,1,2) $ is a weighted complete intersection. 
	We write $ q_2 = q_2\pr + c z$, where $ q_2\in R[x_0,\dots,x_4]_2 $ and $ c \in R $. We set $ B_{hyp} = V(c) \sbe B $. 
	
	The assertion that a morphism $ g\colon T \to B $ is a morphism of schemes such that $ X \times_B T$ is hyperelliptic factors uniquely through $B_{hyp}$ is local on $T$. 
	We may therefore assume that $ T = \Spec S $ is affine. Let $ \alpha\colon R \to S $ be the morphism of rings corresponding to $g$. 
	Now $$ X \times_B T \cong V_+(\alpha(q_4),\alpha(q_2')+\alpha(c)z) \sbe \PP_S(1,1,1,1,1,2) $$ 
	is a presentation as weighted complete intersection. Therefore $\alpha(c) = 0 $ by Lemma \ref{L:definition_does_not_depend_on_cover}. Hence $ g$ factors uniquely though $ B_{hyp} = V(c) = \Spec R/(c)$.
\end{proof}

There is a similar statement for the smooth quartic locus that is proven completely analogous.

\begin{PropDef}(smooth quartic locus)\label{p:smooth_quartic_locus}
	Let $B$ be a scheme over $ \ZZ[1/2] $ and let $ f\colon X \to B $ be a Fano threefold of type (1,1,4). There is an open subscheme $ o\colon B_{sq} \to  B $ such that $ f\colon f\inv(B_{sq}) \to B_{sq}$ is a smooth quartic with the following universal property.
	If $ g\colon T \to B $ is a morphism of schemes such that $ X \times_B T$ is a smooth quartic, then there is a unique morphism $ g\pr\colon T \to B_{sq} $ such that $ g = c\circ g\pr$. 
	Furthermore $ B_{sq} $ is the complement of $ B_{hyp}$ in $ B $.
\end{PropDef}

\subsection*{The involution on a hyperelliptic Fano scheme}

Next, we want to construct an involution on hyperelliptic Fano threefolds. Note that local presentations can be chosen in a certain form. 

\begin{Lemma}\label{L:local_presentation_for_hyperelliptic_Fanos}
	Let $B$ be a scheme over $ \ZZ[1/2] $ and let $ f\colon X \to B $ be a hyperelliptic Fano threefold of type (1,1,4).
	Let $b\in B$ be a point. There is an affine open neighbourhood $ U = \Spec R \sbe B $ of $b$ and polynomials $q_2,q_4 \in R[x_0,\dots,x_4]$ of degree $\deg(q_2) = 2$ and $ \deg(q_4) = 4$  such that there is a isomorphism 
	$$
	(f\inv(U),\omega_{f\inv(U)/U}\inv) \cong (V_+(q_2,q_4-z^2),\Oc(1)) 
	$$
	of polarized schemes over $U$.
\end{Lemma}

\begin{proof}
	By Theorem \ref{t:local_presentation_for_Fanos}, we can find a local presentation $ V_+(q_2,q_4) $ with $q_2,q_4 \in R[x_0,\dots,x_4] $. Write $ q_4 = q_4' + p_2 z + d z^2 $ with 
	$q_4',p_2 \in R[x_0,\dots,x_4] $, $d\in R$. The coefficient $d$ is invertible. Otherwise there would be a fibre with $ d \otimes k(b) = 0 $. This is impossible as this fibre then would be singular by Lemma \ref{l:smooth subvarieties don't contain Q}. After the change of coordinates $ z \mapsto z + \frac{p_2}{2d}$, we get the desired equations.  
\end{proof}

\begin{Def}\label{D:double_cover_flip_involution}
	Let $B$ be a scheme over $ \ZZ[1/2] $ and let $ f\colon X \to B $ be a hyperelliptic Fano threefold of type (1,1,4).
	The \emph{involution induced by the double cover} is the map $ \iota \in \Aut_B(X)$ that over any open affine $ U \sbe B $ with 
	$(f\inv(U),\omega_{f\inv(U)/U}\inv) \cong (V_+(q_2,q_4-z^2),\Oc(1)) $ as in Lemma \ref{L:local_presentation_for_hyperelliptic_Fanos} is given by $ z \mapsto - z$. 
\end{Def}

To see that the involution is well defined, we have to show that the defined maps agree on the intersection of two open affines over which $X$ has a presentation of the given form. Since we can cover the intersection with smaller open affines, it suffices to show the following statement. If $ U = \Spec R \sbe B $ is open affine
and $q_2,q_4,q_2',q_4' \in R[x_0,\dots,x_4]$ are polynomials such that 
$$
(V_+(q_2,q_4-z^2),\Oc(1)) \cong 	(f\inv(U),\omega_{f\inv(U)/U}\inv) \cong (V_+(q_2',q_4'-z^2),\Oc(1)),
$$
then the maps that are given by $ z \mapsto -z $ are identified under the isomorphisms. The isomorphism of polarized schemes induces an isomorphism of graded $R$-algebras
$$
\alpha\colon R[x_0,\dots,x_4,z]/(q_2,q_4-z^2) \to R[x_0,\dots,x_4,z]/(q_2',q_4'-z^2).	
$$
To preserve the equations, the image $\alpha(z) = a z + p$ now has to satisfy $ p = 0$. Hence $\alpha$ commutes with the map $z \mapsto -z$.

%----------------------------------------------------------------------------------------
%	The stack of hyperelliptic Fanothreefolds
%----------------------------------------------------------------------------------------

\section{The stacks of hyperelliptic and smooth quartic Fanos} \label{s:stacks of hyperelliptics and quartics}

In this section, we define the stack of hyperelliptic and smooth quartic Fanos and prove
Theorem {\ref{t:stratification of Fano stack}}.% and Theorem {\ref{t:stack of Fanos is DM}}.

\begin{Def}
	Let $ \mathcal{F} $ over $\ZZ[1/2]$ be the stack of Fano threefolds of type (1,1,4) as defined in Section \ref{s:stack of Fano threefolds}.
	\begin{enumerate}
		\item We define $ \mathcal{H} $ to be the full fibred subcategory of $ \mathcal{F}$ of those Fano threefolds $f\colon X \to B $ that are hyperelliptic.
		\item We define $ \mathcal{Q} $ to be the full fibred subcategory of $ \mathcal{F}$ of those Fano threefolds $f\colon X \to B $ that are smooth quartic.
	\end{enumerate} 
\end{Def}

\begin{Lemma}
	The categories fibred in groupoids $ \mathcal{H}\to \Spec \ZZ[1/2]$ and $ \mathcal{Q}\to \Spec \ZZ[1/2]$ are stacks.
\end{Lemma}

\begin{proof}For the statement on smooth quartics, we refer to \cite{Benoist2013}. We prove the statement about $\mathcal{H}$.
	Since $\mathcal{H}$ is a full subcategory of the stack $ \mathcal{F} $, it is a prestack.
	It remains to show that any descent datum for any given fppf cover is effective. To do so, we consider a cartesian square 
	$$
	\begin{tikzcd}
		X'\ar[r] \ar[d,"f'"] &  X\ar[d,"f"] \\
		B' \ar[r,"g"]  &  B
	\end{tikzcd}
	$$ 
	where $f\colon X \to B $ and $f'\colon X' \to B' $ are Fano threefolds with the latter being hyperelliptic and g is an fppf morphism. Let $ B = \bigcup_{i\in I} U_i $, $U_i = \Spec R_i$ be an open affine cover such that for each $ i\in I $ there are $ q_4 \in R_i[x_0,\dots,x_4,z]_4$, $q_2\in R_i[x_0,\dots,x_4]_2 $ and $ c\in R_i$ such that 
	$$
	(V_+(q_2+cz,q_4),\Oc(1)) \cong 	(f\inv(U_i),\omega_{f\inv(U_i)/U_i}\inv). 
	$$
	Now we can find a cover $ B' = \bigcup_{i\in I,j\in J_i} V_{i,j} $,
	where $V_{i,j}$ is quasi compact and $ g(V_{i,j}) = U_i $. Fix some $i,j$. Let $ V_{i,j} = \bigcup_{m=1}^n W_m$ be a finite open affine cover. Set $ W = \coprod_{m=1}^n W_m $.
	Then $ W $ is affine, $X' \times_{B'} W $ is hyperelliptic and the map $ W \to U_i$ induced  by $g$ is faithfully flat. The corresponding ring map 
	$$ \alpha\colon R_i \to S := \Gamma(W,\Oc_W) $$ is faithfully flat, hence injective. As 
	$$X' \times_{B'} W \cong V_+(\alpha(q_2)+\alpha(c)z,\alpha(q_4)) \sbe \PP_S(1,1,1,1,1,2)$$ is hyperelliptic we see that $\alpha(c)=0$ by Lemma \ref{L:local_presentation_for_hyperelliptic_Fanos}. Hence $ c = 0 $. As $i,j$ were chosen arbitrarily, this proves that $f\colon X \to B $ is hyperelliptic.
\end{proof}

\begin{proof}[Proof of Theorem {\ref{t:stratification of Fano stack}}]
	That $ \Hc \to \Fc $ is a representable closed immersion and $ \mathcal{Q} \to \Fc $ is a representable open immersion is a reformulation of Proposition \ref{p:hyperelliptic_locus} and Proposition \ref{p:smooth_quartic_locus}, respectively. For details, see \cite{LichtPdh2022}. That for any field $k$ with $2\in k^\times $, we have $ \mathcal{F}(k) = \mathcal{Q}(k) \sqcup \mathcal{H}(k) $ follows from Proposition \ref{p:classifying_Fanos_over_nonclosed_field}.
\end{proof}

%----------------------------------------------------------------------------------------
%	Deformations of hyperelliptic Fanothreefolds
%----------------------------------------------------------------------------------------

\section{First-order deformations of hyperelliptic Fano threefolds of type (1,1,4)}

In this section, we determine the space of first-order deformations of hyperelliptic Fano threefolds of type (1,1,4).
Let $k$ be a field with $2\in k^\times$, let $k[\epsilon] := k[x]/(x^2)$ be the ring of dual numbers over $k$ and let $ f\colon X \to \Spec k $ be a hyperelliptic Fano threefold of type (1,1,4). 
A \emph{first-order lift} of $ X$ is given by a pair of morphisms 
$$ (\tilde{f}\colon \tilde{X} \to \Spec k[\epsilon], i\colon X \to \tilde{X})$$ such that the square 
$$
\begin{tikzcd}
	X \ar[r,"i"] \ar[d,"f"] &   \tilde{X} \ar[d,"\tilde{f}"] \\
	\Spec k \ar[r,"g"]  &  \Spec k[\epsilon]
\end{tikzcd}
$$
is cartesian, where $ g $ is the closed immersion given by $ \epsilon \to 0 $.
Two lifts $(\tilde{f}'\colon \tilde{X}' \to \Spec k[\epsilon], i'\colon X \to \tilde{X}) $ and $(\tilde{f}\colon \tilde{X} \to \Spec k[\epsilon], i\colon X \to \tilde{X})$ are considered to be equivalent if there is an isomorphism $\phi\colon \tilde{X}' \to \tilde{X}$ over $ k[\epsilon] $ such that $  \phi \circ i' = i $. An equivalence class of first-order lifts is called a \emph{first-order deformation}.

We are interested in characterizing first-order deformations of $X$ which are again hyperelliptic.
For this note that $X$ comes with an involution associated to the double cover $ \iota\colon X \to X $; see Definition \ref{D:double_cover_flip_involution}.
If $(\tilde{f}\colon \tilde{X} \to \Spec k[\epsilon], i\colon X \to \tilde{X})$ is a first-order lift such that $ \tilde{f}\colon \tilde{X} \to \Spec k[\epsilon] $ is hyperelliptic, then there is an involution $ \tilde{\iota}\colon \tilde{X} \to \tilde{X} $ which extends $\iota$. From the explicit description of $X$ as a weighted complete intersection it follows that a lift is hyperelliptic if and only if the involution of $X$ extends to the lift. 
This observation allows us to determine the space of deformations that are hyperelliptic.

\begin{Th}\label{t:deformations of hyperelliptic Fanos}
	If $X$ is a hyperelliptic Fano threefold of type $(1,1,4)$ over a field $k$ with $2\in k^\times$  with involution $\iota$, then the space of first-order deformations of $X$ that are hyperelliptic is the subspace of $\iota$-invariants  
	$$
	\Homology^1(X,T_X)^\iota \sbe \Homology^1(X,T_X)
	$$ 
	of the space of all first-order deformations. 
\end{Th}

\begin{proof}%[Proof of Theorem {\ref{t:deformations of hyperelliptic Fanos}}]
	Let $ G = \ZZ/2\ZZ $. Then $G$ acts on $X$ via $\iota$. By the observation above and \cite[Proposition~3.2.7]{Bys09}, the space of first-order deformations of $X$ that are hyperelliptic is given by the equivariant sheaf cohomology group $ \Homology^1(X;G,T_X) $. By \cite[Theorem~1.4]{LichtTorelli2022}, the action of $G$ on $ T_X $ is faithful. Therefore we can apply \cite[Proposition~5.2.3]{Gro57} to see that $ \Homology^1(X;G,T_X) = \Homology^1(X,T_X)^G = \Homology^1(X,T_X)^\iota$.
\end{proof}

%----------------------------------------------------------------------------------------
%	Arithmetic hyperbolicity
%----------------------------------------------------------------------------------------

%----------------------------------------------------------------------------------------
%	Geometric hyperbolicity
%----------------------------------------------------------------------------------------

\section{Geometric Hyperbolicity and Persistence}

Let $k$ be an algebraically closed field of characteristic zero. We follow \cite[\S 2]{JLIntegralPoints2019} and say that  a finitely presented algebraic stack $X$  over  $k$ is   \emph{geometrically hyperbolic over $k$} if,  for every smooth integral curve $C$ over $k$, every point   $c\in C(k)$ and object $ x \in X(k)$, the set  of isomorphism classes of morphisms $ f\colon C \to X$ with $ f(c) \cong x $ is finite. 
We will show that geometric hyperbolicity implies the persistence of arithmetic hyperbolicity. The proof of this fact uses an inductive argument. The following lemma gives us the induction step.

\begin{Lemma}\label{l:arith hyp + geo hyp implie trdeg 1 extesions are arith hyp}
	Let $ k \sbe L$ be an extension of algebraically closed fields of characteristic zero such that $L$ is of transcendence degree $1$ over $k$. Let $X$ be a finite type separated arithmetically hyperbolic Deligne-Mumford stack over $k$.
	If $ X $ is geometrically hyperbolic over $k$, then $X_L$ is arithmetically hyperbolic over $L$. 
\end{Lemma}

\begin{proof}
	(If $X$ is a variety, then this is \cite[Lemma~4.2]{Javanpeykar2021}. We adapt the proof of \emph{loc. cit.} to the setting of stacks.)
	
	The notion of arithmetic hyperbolicity is independent of the chosen model; check for \cite[Lemma~4.8]{JLChevalleyWeil2020}.
	We choose a $\ZZ$-finitely generated subring $ A\sbe k $ and a model $\Xc$ for $ X $ over $A$. 
	%After replacing $A$ by a finitely generated extension if necessary, we may and do assume that it is smooth over $\ZZ$. 
	%We use the criterion given in Proposition \ref{l:bare bones criterion for arithmetic hyperbolicity} to show that $X_L$ is arithmetically hyperbolic. 
	Since $X$ is separated and Deligne-Mumford, it has a finite diagonal. 
	The property of having a finite diagonal spreads out; see \cite[B.3]{Rydh2015}. 
	Hence after possibly replacing $A$ with a finitely generated extension contained in $K$, we may and do assume that $\Xc$ has a finite diagonal.  
	Note that $\Xc$ is also a model for $ X_L $.
	
	Let $ B \sbe L $ be a $\ZZ$-finitely generated subring containing $A$. We will find a $\ZZ$-finitely generated subring $B' \sbe L$ containing $B$ such that $\pi_0(\Xc(B'))$ is finite.
	We may assume $B$ is not contained in $ k $. Otherwise, we are done since $X$ is arithmetically hyperbolic over $k$. 
	
	The  morphisms $ \Spec B  \to \Spec A$ and $\Spec A \to \Spec \ZZ$ are generically smooth and finitely presented. As smoothness spreads out, we can find finitely generated extensions $ A \sbe A' \sbe k $ and $ B \sbe B' \sbe L $ such that $A' \sbe B'$ and both $ \Spec B'  \to \Spec A'$ and $\Spec A' \to \Spec \ZZ $ are smooth.
	The scheme $\mathcal{C} = \Spec B'$ is integral and smooth of relative dimension $1$ over $  \Spec A' $.
	As $k$ is algebraically closed, the affine curve $ C = \mathcal{C}_k $ has $k$-sections. 
	Hence, after possibly replacing $A'$ and $B'$ by finitely generated extensions still satisfying $ A' \sbe k $ and $ B' \sbe L$ and the smoothness properties above, there is a section $ c \in  \mathcal{C}(A') $. 
	Since $A'$ is smooth over $ \ZZ$, it is in particular integrally closed. Therefore, as $ X $ is arithmetically hyperbolic over $k$,
	by applying \cite[Theorem~4.23]{JLChevalleyWeil2020}, we see that 
	$ \pi_0(\Xc(A')) $ is finite.    
	For any morphism $f\colon \mathcal{C} \to \Xc $, we have $ f(c) \in \Xc(A') $. Therefore, we get an inclusion
	$$
	\pi_0 (\Xc(B'))=\pi_0 (\Xc(\mathcal{C})) = \Hom_\mathcal{C}(\mathcal{C},\Xc) \sbe \bigcup_{x \in \pi_0(\Xc(A'))}  \Hom_\mathcal{C}((\mathcal{C},c),(\Xc,x)).
	$$
	Furthermore, consider the diagram:
	$$
	\begin{tikzcd}
		\Hom_\mathcal{C}((\mathcal{C},c),(\Xc,x)) \ar[r,"\alpha"] \ar[d,"\iota"] &  \Hom_k((C,c_k),(X,x_k)) \ar[d,] \\
		\Hom_\mathcal{C}(\mathcal{C},\Xc) = \pi_0(\Xc(B')) \ar[r,"\beta"]  &  \pi_0(\Xc(L)) = \Hom(L,X)
	\end{tikzcd}
	$$
	The natural inclusion map $\iota$ is injective and the pullback map $\beta$ has finite fibres by \cite[Proposition~4.19]{JLChevalleyWeil2020}. Hence $\alpha$ has finite fibres. Since $X$ is geometrically hyperbolic over $k$, the set $\Hom_k((C,c_k),(X,x_k))$ is finite. From this, we deduce that $\pi_0 (\Xc(B'))$ is finite.	   
\end{proof}

\begin{Lemma}\label{l:geo hyp implies arith hyp}
	Let $ k \sbe L$ be an extension of algebraically closed fields of characteristic zero and let $X$ be a finite type separated arithmetically hyperbolic Deligne-Mumford stack over $k$ such that, for every algebraically closed subfield $K \sbe L$ containing $k$,  the stack $X_K $ is geometrically hyperbolic over $K$. Then   $X_L$ is arithmetically hyperbolic over $L$.  
\end{Lemma}

\begin{proof} 
	As arithmetic hyperbolicity expresses the finiteness of points valued in  $\ZZ$-finitely generated integral domains of characteristic zero,
	the stack $X_L$ is arithmetically hyperbolic if and only if $ X_K $ is arithmetically hyperbolic for all algebraically closed subfields $ K \sbe L$ that have a finite transcendence degree over $k$. Therefore, we may and do assume that $L/k$ has a finite transcendence degree.
	Let $ k = K_0 \sbe K_1 \sbe \dots \sbe K_n = L$ be a chain of extensions of algebraically closed fields each having transcendence degree one. We apply Lemma \ref{l:arith hyp + geo hyp implie trdeg 1 extesions are arith hyp} inductively to see that $ X_{K_i} $ is arithmetically hyperbolic for all $i$.
\end{proof}

\begin{Th}\label{t:geometric hyp implies algebraic hyp presists}
	Let $k$ be an uncountable algebraically closed field of characteristic zero. Let $X$ be a finite type separated Deligne-Mumford stack over $k$. If $X$ is arithmetically hyperbolic over $k$ and geometrically hyperbolic over $k$, then $ X_L/L$ is arithmetically hyperbolic over $L$ and geometrically hyperbolic over $L$ for any algebraically closed field extension $k\sbe L$.
\end{Th}

\begin{proof}
	This is a combination of Lemma \ref{l:geo hyp implies arith hyp} with the fact that geometric hyperbolicity over uncountable fields persists   \cite[Lemma~2.4]{JLIntegralPoints2019}.
\end{proof}

The following two lemmata will be used to prove that $\Fc$ is geometrically hyperbolic. 

\begin{Lemma}\label{l:X(C) injects X(k(C))}
	Let $X$ be a finite type separated  Deligne-Mumford stack over a field $k$ and let $C$ be a smooth integral curve over $k$ with function field $K$. Then the map $ \pi_0(X(C)) \to \pi_0 (X(K))$ is injective. 
	%I.e. if for two morphisms 
\end{Lemma}

\begin{proof}
	Let $ g_1,g_2 \in X(C)$ be objects.
	Since the diagonal $ \Delta\colon X \to X \times_k X $ is finite, the Isom-sheaf $\mathrm{Isom}_{X(C)}(g_1,g_2)$ is a finite scheme over $C$.
	In particular, every generic section of $$\mathrm{Isom}_{X(C)}(g_1,g_2) \to C $$ extends by the valuative criterion of properness (uniquely) to a section.  
	I.e. if $g_1,g_2$ are isomorphic over $K$, then they are isomorphic over $C$.
\end{proof}

\begin{Lemma}\label{l:q finite preimage of geo hyp is geo hyp}
	Let $ f\colon X \to Y $ be a quasi-finite representable morphism of separated finite type Deligne-Mumford stacks over a field $k$. If $Y$ is geometrically hyperbolic, then $X$ is geometrically hyperbolic.
\end{Lemma}

\begin{proof}
	Let $C$ be a smooth integral curve over $k$ and $ c \in C(k) $, $x \in X(k) $ points. Let $ y = f(x) \in Y(k) $. The set $ \Hom_k((C,c),(Y,y)) $ is finite. Therefore it suffices to show that for any fixed $ g \in \Hom_k((C,c),(Y,y)) $, there are only finitely many isomorphism classes of morphisms $ \phi\colon (C,c) \to (X,x) $ such that $ f \circ \phi \cong g $. Consider the generic point $ \eta\colon \Spec k(C) \to C $. As $f$ is quasi-finite, there are only finitely many possibilities for the image of $\eta$ in $ X $ up to isomorphism. Hence by Lemma \ref{l:X(C) injects X(k(C))} there are only finitely many possibilities for the isomorphism class of $\phi$. 
\end{proof}

%----------------------------------------------------------------------------------------
%	Period map and arithmetic hyperbolicity
%----------------------------------------------------------------------------------------

\section{Period map and arithmetic hyperbolicity}

In this section, we will proof Theorem \ref{t:Fano stack is arithmetically hyperbolic} and Theorem \ref{t:Shaferevic for Fanos}.
The following well-known lemma will be used to prove the arithmetic hyperbolicity of the stack $\Fc$.

\begin{Lemma}\label{l:inj on tangent spaces implies q fininite}
	Let $f\colon X \to Y $ be a morphism of finite type separated Deligne-Mumford stacks. If $f$ is injective on tangent spaces, then $f$ is quasi-finite. 
\end{Lemma}

\begin{proof} 
	Let $Y' \to Y $ be an étale surjective morphism with $ Y' $ a scheme and let $ X' \to X \times_Y Y'$ be an an étale surjective morphism with $ X'$ a scheme.
	Then the morphism $ X' \to Y' $ induced by $f$ is injective on tangent spaces. Hence it is unramified and therefore quasi-finite; see \cite[Tag~02V5,Tag~02BG]{stkpr}. This shows $f$ is quasi-finite.	
\end{proof}

Let $\mathcal{A}_g$ be the stack of principally polarized abelian varieties of dimension $g$. Recall that $\mathcal{A}_g$ is a finite type separated Deligne-Mumford stack over $\ZZ$; see \cite{Olsson2008}. Over the complex numbers, the intermediate Jacobian of a Fano threefold defines a complex-analytic period map to $ \Ac_{g,\CC}^{an} $.
Our main result is that this period map has finite fibres in the setting of Fano threefolds of type $(1,1,4)$.  
\begin{Th}\label{t:period map is q finite}
	The period map $$ p:\mathcal{F_\CC}^{an} \to \Ac_{30,\CC}^{an} $$ is quasi-finite.
\end{Th}

\begin{proof}%[Proof of Theorem \ref{t:period map is q finite}]
	First, we consider the period map restricted to the smooth quartic locus.
	The infinitesimal Torelli problem for hypersurfaces is completely understood; see \cite{CaG80}. In particular, if $X$ is a smooth quartic, the differential of the period map 
		$$(dp)_X\colon 
	\Homology^1(X,T_X) \to \bigoplus_{p+q = n} \Hom_\CC\left( H^p(X,\Omega_X^q), H^{p+1}(X,\Omega_X^{q-1}) \right)
	$$ 
	is injective. Therefore, the period map restricted to $\mathcal{Q}$ is quasi-finite by Lemma \ref{l:inj on tangent spaces implies q fininite}.
	
	Furthermore, by Theorem \ref{t:deformations of hyperelliptic Fanos},
	if $X$ is a hyperelliptic Fano threefold over $\CC$ with involution $\iota$ asscociated to the double cover, then the tangent space of $\mathcal{H}$ at the object $X$ is $ \Homology^1(X,T_X)^\iota $. Therefore, by \cite[Theorem~1.3]{LichtTorelli2022}, the differential of the period map restricted to $ \mathcal{H}$ is injective. Again by Lemma \ref{l:inj on tangent spaces implies q fininite}, we conclude that $ p\vert_\Hc$ is quasi-finite. 
	
	Since $\mathcal{F}$ is the union of $\mathcal{Q}$ and $\Hc$ (Theorem \ref{t:stratification of Fano stack}), the period map $p$ is quasi-finite. 
\end{proof}

\begin{proof}[Proof of Theorem \ref{t:Fano stack is arithmetically hyperbolic}]
	By \cite{Faltings1984}, the stack $ (\Ac_{30})_{\oQ} $ is absolutely arithmetically hyperbolic. 
	Moreover, by \cite[Theorem~1.7]{JLIntegralPoints2019}, the stack $(\Ac_{30})_\CC$ is geometrically hyperbolic. 
	Note by Proposition \ref{P:facts about the stack of Fanos}.(5) the stack $\Fc_\CC$ is Deligne-Mumford.
	Hence by \cite[Theorem~6.4]{JLChevalleyWeil2020}, Lemma \ref{l:q finite preimage of geo hyp is geo hyp} and Theorem \ref{t:period map is q finite},
	the stack $\Fc_\CC$ is arithmetically hyperbolic and geometrically hyperbolic. Hence by Theorem \ref{t:geometric hyp implies algebraic hyp presists}, the stack $X_L$ is arithmetically and geometrically hyperbolic for any algebraically closed field extension $ L\spe \CC $. 
\end{proof}

\begin{proof}[Proof of Theorem \ref{t:Shaferevic for Fanos}]
	Let $K$ be a number field and let $S$ be a finite set of places on $K$. We may assume that all places lying over $2$ are contained in $S$.
	As seen in \cite[Proposition~2.10]{JLo18}, if $B$ is a connected scheme and $f:X\to B$ is a Fano threefold, then the type is constant in the fibres of $f$.
	Hence, the set of isomorphism classes of Fano threefolds $X$  of type $(1,1,4)$ over $K$ with good reduction outside $S$ is the image
	$$
	\mathrm{im} \left(\pi_0(\Fc(\OKs)) \to \pi_0(\Fc(K))\right). 
	$$
	Since $ \Fc_{\oQ} $ is arithmetically hyperbolic by Theorem \ref{t:Fano stack is arithmetically hyperbolic}, this set is finite by  \cite[Theorem~4.22]{JLChevalleyWeil2020}.
\end{proof}

%----------------------------------------------------------------------------------------
%	BIBLIOGRAPHY
%----------------------------------------------------------------------------------------

\printbibliography[heading=bibintoc]
\end{document}